\documentclass[english]{article}

\usepackage[T1]{fontenc}
\usepackage[utf8]{inputenc}
\usepackage{babel}

\usepackage{color}

\usepackage{graphicx}
\usepackage{amsmath}
\usepackage{amsthm}
\usepackage{amssymb}
\usepackage{amsfonts}
\usepackage{latexsym}

\definecolor{blue}{rgb}{0,0,0.8}
\definecolor{red}{rgb}{0.8,0,0}
\definecolor{darkgreen}{rgb}{0,0.6,0}

\newcommand{\PP}{{\mathcal P}}
\newcommand{\ze}{\zeta}
\newcommand{\zs}{\zeta(s)}
\newcommand{\DD}{{\mathcal D}}
\newcommand{\GG}{{\mathcal G}}

\newcommand{\si}{\sigma}

\newcommand{\G}{{\mathcal G}}
\newcommand{\A}{{\mathcal A}}

\newcommand{\Z}{{\mathcal Z}}

\newcommand{\RR}{{\mathbb R}}
\newcommand{\CC}{{\mathbb C}}
\newcommand{\KK}{{\mathbb K}}

\newcommand{\NN}{{\mathbb N}}

\newcommand{\N}{{\mathcal N}}
\newcommand{\R}{{\mathcal R}}

\newcommand{\al}{{\alpha}}
\newcommand{\ve}{{\varepsilon}}

\newtheorem{theorem}{Theorem}

\newtheorem{lemma}{Lemma}

\newtheorem{definition}{Definition}

\theoremstyle{definition}

\newtheorem{remark}{Remark}

\newcommand{\de}{\delta}

\begin{document}

\title
{Density estimates for the zeros of the \\ Beurling $\zeta$ function in the critical strip\thanks{Supported in part by Hungarian National Foundation for Scientific Research, Grant \# T-72731,
T-049301, K-61908, K-119528 and K-132097 and by the Hungarian-French Scientific and
Technological Governmental Cooperation, Project \# T\'ET-F-10/04 and the Hungarian-German Scientific and Technological Governmental Cooperation, Project \# TEMPUS-DAAD \# 308015.}}

\author{Szil\' ard Gy. R\' ev\' esz}

\date{\it{Dedicated to J\'anos Pintz on the occasion of his 70th anniversary}}

\maketitle

\begin{abstract}
In this paper we prove three results on the density resp. local density and clustering of zeros of the Beurling zeta function $\zs$ close to the one-line $\si:=\Re s=1$. The analysis here brings about some news, sometimes even for the classical case of the Riemann zeta function.

As a complement to known results for the Selberg class, first we prove a Carlson type zero density estimate. Note that density results for the Selberg class rely on use of the functional equation of $\zeta$, not available in the Beurling context. Our result sharpens results of Kahane, who proved an $O(T)$ estimate for zeros lying precisely just on a vertical line $\Re s=a$ in the critical strip.

Next we deduce a variant of a well-known theorem of Tur\'an, extending its range of validity even for rectangles of height only $h=2$.

Finally, we extend a zero clustering result of Ramachandra from the Riemann zeta case. A weaker result -- which, on the other hand, is a strong sharpening of the average result from the classic book of Montgomery -- was worked out by Diamond, Montgomery and Vorhauer.
On our way we show that some obscure technicalities of the Ramachandra paper can be avoided.
\end{abstract}

{\bf MSC 2000 Subject Classification.} Primary 11M41; Secondary 11F66, 11M36, 30B50, 30C15.

{\bf Keywords and phrases.} {\it Beurling zeta function, analytic
continuation, arithmetical semigroups, Beurling prime number
formula, zero of the Beurling zeta function, oscillation of remainder term, density estimates for zeta zeros.}

\medskip
{\bf Author information.} Alfréd Rényi Institute of Mathematics\\
Reáltanoda utca 13-15, 1053 Budapest, Hungary \\
{\tt revesz.szilard@renyi.hu}

\section{Introduction}

\subsection{Beurling's theory of generalized integers and primes.}
Beurling's theory fits well to the study of several mathematical
structures. A vast field of applications of Beurling's theory is
nowadays called \emph{arithmetical semigroups}, which are
described in detail e.g. by Knopfmacher, \cite{Knopf}.

Here $\G$ is a unitary, commutative semigroup, with a countable set
of indecomposable generators, called the \emph{primes} of $\G$ and
denoted usually as $p\in\PP$, (with $\PP\subset \G$ the set of all
primes within $\G$), which freely generate the whole of $\G$: i.e.,
any element $g\in \G$ can be (essentially, i.e. up to order of
terms) uniquely written in the form $g=p_1^{k_1}\cdot \dots \cdot
p_m^{k_m}$: two (essentially) different such expressions are
necessarily different as elements of $\G$, while each element has
its (essentially) own unique prime decomposition.
Moreover, there is a \emph{norm} $|\cdot|~: \G\to \RR_{+}$ so that
the following hold. First, the image of $\G$, $|\G|\subset
\RR_{+}$ is discrete, i.e. any finite interval of $\RR_{+}$ can contain
the norm of only a finite number of elements of $\G$; thus the
function
\begin{equation}\label{Ndef}
{\N}(x):=\# \{g\in \G~:~ |g| \leq x\}
\end{equation}
exists as a finite, nondecreasing, right continuous, nonnegative integer valued
function on $\RR_{+}$.
Second, the norm is multiplicative, i.e. $|g\cdot h| = |g| \cdot
|h|$; it follows that for the unit element $e$ of $\G$ $|e|=1$, and
that all other elements $g \in \G$ have norms strictly larger than 1.

Arithmetical functions can also be defined on
$\G$. We will use in this work the
identically one function $u$, the M\"obius function $\mu$ and the
number of divisors function $d$: for their analogous to the classical case
definition see pages 73-79 in \cite{Knopf}. The generalized von Mangoldt function $\Lambda_{\GG}(g)$, will appear below in \eqref{vonMangoldtLambda}.


In this work we assume the so-called \emph{"Axiom A"} (in its normalized form to $\delta=1$) of
Knopfmacher, see pages 73-79 of his fundamental book \cite{Knopf}.

\begin{definition} It is said that ${\N}$ (or, loosely speaking, $\ze$)
satisfies \emph{Axiom A} -- more precisely, Axiom
$A(\kappa,\theta)$ with the suitable constants $\kappa>0$ and
$0<\theta<1$ -- if we have\footnote{The usual formulation uses the more natural version $\R(x):= \N(x)-\kappa x$. However, our version is more convenient with respect to the initial values at 1, as we here have $\R(1-0)=0$. All respective integrals of the form $\int_X$ will be understood as integrals from $X-0$, and thus we can avoid considering endpoint values in the partial integration formulae involving integration starting form 1. Alternatively, we could have taken also $\N(x):=\# \{g\in \GG, |g|<x\}$ left continuous, and $\R(x):=\N(x)-\begin{cases}\kappa x \qquad &\text{if}~ x>1 \\ 0 &\text{if} ~ x\le 1 \end{cases}$--also with this convention we would have $\R(1-0)=0$ for the remainder term, but this seemed to be less convenient than our above choice.} for the remainder term
$$
\R(x):= \N(x)-\kappa (x-1)
$$
the estimate
\begin{align}\label{Athetacondi}
\left| \R(x) \right|  \leq A x^{\theta} \quad ( x \geq 1 ).
\end{align}
\end{definition}

It is clear that under Axiom A the Beurling zeta function
\begin{equation}\label{zetadef}
\ze(s):=\ze_{\G}(s):=\int_1^{\infty} x^{-s} d\N(x) = \sum_{g\in\G} \frac{1}{|g|^s}
\end{equation}
admits a meromorphic, essentially analytic continuation
$\kappa\frac{1}{s-1}+\int_1^{\infty} x^{-s} d\R(x)$ up to $\Re s >\theta$
with only one, simple pole at 1.

\subsection{Analytic theory of the distribution of Beurling generalized primes.}
The Beurling $\zeta$ function \eqref{zetadef} is expressed by
the generalized von Mangoldt function
\begin{equation}\label{vonMangoldtLambda}
\Lambda (g):=\Lambda_{\G}(g):=\begin{cases} \log|p| \quad \textrm{if}\quad g=p^k,
~ k\in\NN ~~\textrm{with some prime}~~ p\in\G\\
0 \quad \textrm{if}\quad g\in\G ~~\textrm{is not a prime power in} ~~\G
\end{cases},
\end{equation}
as coefficients of the logarithmic derivative of the zeta function
\begin{equation}\label{zetalogder}
-\frac{\zeta'}{\zeta}(s) = \sum_{g\in \G} \frac{\Lambda(g)}{|g|^s}.
\end{equation}

The Beurling theory of generalized primes is mainly concerned with
the analysis of the summatory function
\begin{equation}\label{psidef}
\psi(x):=\psi_{\G}(x):=\sum_{g\in \G,~|g|\leq x} \Lambda (g).
\end{equation}
The generalized PNT (Prime Number Theorem) is the asymptotic equality $\psi(x)\thicksim x$. The remainder term in this equivalence is denoted, as usual,
\begin{equation}\label{Deltadef}
\Delta(x):=\Delta_{\G}(x):=\psi(x)-x.
\end{equation}
In the classical case of prime number distribution, as well as regarding some extensions to primes in arithmetical progressions and distribution of prime ideals in algebraic number fields, the connection between location and distribution of zeta-zeroes and oscillatory behavior in the remainder term of the prime number formula $\psi(x)\thicksim x$ is well understood \cite{Kac1, K-97, K-98, Knapowski, Pintz1, Pintz2, Pintz9, PintzProcStekl, Rev1, Rev2, RevPh, RevAA, Stas1, Stas2, Stas-Wiertelak-1, Stas-Wiertelak-2, Turan1, Turan2}. On the other hand in the generality of Beurling primes and zeta function, investigations so far were focused on mainly four directions. First, better and better, minimal conditions were sought in order to have a Chebyshev type formula $x\ll \psi(x) \ll x$, see e.g. \cite{Vindas12, Vindas13, DZ-13-2, DZ-13-3}. Understandably, as in the classical case, this relation requires only an analysis of the $\zeta$ function of Beurling in, and on the boundary of the convergence halfplane. Second, conditions for the prime number theorem to hold, were sought see e.g. \cite{Beur, K-98, DebruyneVindas-PNT, DSV, DZ-17, Zhang15-IJM, Zhang15-MM}. Again, this relies on the boundary behavior of $\zeta$ on the one-line $\si=1$. Third, rough (as compared to our knowledge in the prime number case) estimates and equivalences were worked out in the analysis of the connection between $\zeta$-zero distribution and error term behavior for $\psi(x)$ see e.g. \cite{H-5}, \cite{H-20}. Fourth, examples were constructed for arithmetical semigroups with very "regular" (such as satisfying the Riemann Hypothesis RH and error estimates $\psi(x)=x+O(x^{1/2+\varepsilon})$) and very "irregular" (such as having no better zero-free regions than \eqref{classicalzerofree} below and no better asymptotic error estimates than \eqref{classicalerrorterm}) behavior and zero- or prime distribution, see e.g. \cite{H-15}, \cite{BrouckeDebruyneVindas}, \cite{DMV}, \cite{H-5}, \cite{Zhang7}. Here we must point out that the above citations are just examples, and are far from being a complete description of the otherwise formidable literature\footnote{E.g. a natural, but somewhat different direction, going back to Beurling himself, is the study of analogous questions in case the assumption of Axiom A is weakened to e.g. an asymptotic condition on $\N(x)$ with a product of $x$ and a sum of powers of $\log x$, or sum of powers of $\log x$ perturbed by almost periodic polynomials in $\log x$, or $\N(x)-cx$ periodic, see \cite{Beur}, \cite{Zhang93}, \cite{H-12}, \cite{RevB}.}. For a throughout analysis of these directions as well as for much more information the reader may consult the monograph \cite{DZ-16}.

The main focus of our study, of which this present paper is a part only, is to establish as precise as possible connections between distribution of the zeros of the Beurling zeta function $\zeta$ on the one hand and order of magnitude estimates or oscillatory properties of $\Delta(x)$ on the other hand.

Apart from generality and applicability to e.g. distribution of
prime ideals in number fields, the interest in these things were
greatly boosted by a construction of Diamond, Montgomery
and Vorhauer \cite{DMV}. They basically showed that under Axiom A
the Riemann hypothesis may still fail; moreover, nothing better
than the most classical zero-free region and error term \cite{V} of
\begin{equation}\label{classicalzerofree}
\zeta(s) \ne 0 \qquad \text{whenever}~~~ s=\sigma+it, ~~ \sigma >
1-\frac{c}{\log t},
\end{equation}
and
\begin{equation}\label{classicalerrorterm}
\psi(x)=x +O(x\exp(-c\sqrt{\log x})
\end{equation}
follows from \eqref{Athetacondi} at least if $\theta>1/2$.

\subsection{Aims and results of the paper.}

The present paper is the second part of a series. In \cite{Rev-MP} we worked out in detail a number of technical lemmas on the behavior of the Beurling zeta function, and arrived at a Riemann-von Mangoldt type formula.
Here we analyze further the distribution of zeroes of the Beurling zeta function in the critical strip $\theta<\Re s \le 1$.

Our aim with this analysis is to lay the ground for the extension to the Beurling case of a number of results of number theory nature. In Section \ref{sec:preview} we will briefly explain what concrete goals we have in mind, but we consider that a lot of further number theory results--like e.g. estimates for primes in short intervals etc.--become also accessible by use of the here presented information on the distribution of the zeroes of the Beurling $\zeta$ function.

We start the present paper with a classical "Carlson type" density result. Theorem \ref{th:density} is somewhat surprising, because we lack a functional equation, essential in the treatment of the Selberg class, where zero density estimates are known to hold \cite{KP}. However, the functional equation is only used in the Selberg class to estimate $\zeta$, and we succeed because similar estimates can be derived directly from our conditions. A predecessor of our result was worked out by Kahane \cite{K-99}, who proved that the number of Beurling zeta zeroes lying on some vertical line $\Re s=\si=a$, where $a>\max(1/2,\theta)$, has finite upper density. That is already a nontrivial fact\footnote{This particular result enabled Kahane to draw deep number theoretical consequences regarding the oscillation (sign changes) of the error term in the prime number formula. Obviously, obtaining a much sharper result -- estimating the total number of zeroes in a full rectangle, not only on one individual vertical line, and with a quantity essentially below the order of $T$ when $a$ is getting close to $1$ -- provides an even stronger foothold for deriving number theoretical consequences.} because the total number of zeroes in the rectangle $[a,1]\times[-iT,iT]$ may grow in the order $T\log T$ for any $a<1/2$.

Second, we present a Turán type local density estimate in case the Beurling zeroes locally somewhat keep off the 1-line. In this we improve upon the local precision allowing heights of rectangles in focus just being constants (instead of the classical $\log\log |t|$).

Finally, the third result is a zero clustering estimate, providing an improvement on the corresponding version of what Diamond, Montgomery and Vorhauer proved and used in \cite{DMV}, which in itself was an improvement over the weaker, averaged result in \cite{Mont}. Our result is a better presented and generalized variant of a classical, but obscurely written and thus seemingly forgotten result of Ramachandra \cite{Ram} for the Riemann zeta case.

The very fact that the proofs of these go through in this generality is somewhat surprising. Moreover, the last two of them contain some refinements even for the Riemann zeta function, so that they bear some novelty, however minor, even in the most classical case.

For deriving the below density theorem we will use the following two additional assumptions, both quite frequent and general, but still forming some restrictions to our general treatment. One very generally used condition is that the norm would actually map to the natural integers. In cases of counting type problems, as well as e.g. for algebraic number fields where certain indices are used as norms (equivalence classes modulo an ideal, e.g.), this is all self-evident.

\begin{definition}[Condition B]\label{condB} We say that
\emph{Condition B} is satisfied, if $|\cdot|:\G\to\NN$, that is,
the norm $|g|$ of any element $g\in\G$ is a natural number.
\end{definition}

As is natural, we will write $\nu\in|\G|$ if there exists $g\in\G$
with $|g|=\nu$. Under Condition B we can introduce the
arithmetical function $G(\nu):=\sum_{g\in\G,~|g|=\nu} 1$, which is
then a super-multiplicative arithmetical function on $\NN$. The next condition is a
kind of "average Ramanujan condition" for the Beurling zeta function.

\begin{definition}[Condition G]\label{condG} We say that
\emph{Condition G} is satisfied, if with a certain $p>1$ we have
for the function
\begin{equation}\label{Fdef}
F_p(X):=\frac1X\sum_{g\in\G ; |g|\leq X} G(|g|)^p = \frac1X
\sum_{\nu \in|\G| ; \nu \leq X} G(\nu)^{1+p}=\frac1X \int_1^X
G^{p}(x)d\N(x)
\end{equation}
the property that
\begin{equation}\label{Gpcondi}
\log F_p(X) = o(\log X) \qquad (X\to \infty),
\end{equation}
that is, for any fixed $\varepsilon >0$
$F_p(X)=O(X^{\varepsilon})$.
\end{definition}

Note that in case $\log G(\nu) = o(\log \nu)$, i.e. when for all
$\ve>0$ we have $G(\nu)=O(\nu^{\ve})$, then Condition G is
automatically satisfied for all $0<p<\infty$. Even this stronger $O(\nu^{\ve})$ order
estimate is proved for many important cases, see e.g. V.2.4. Theorem
and V.2.5. Corollary of \cite{Knopf}. For some discussion of these conditions see 
\S \ref{sec:density}.

We denote the number of zeroes of the Beurling zeta function in $[b,1]\times [-iT,iT]$ as
\begin{equation}\label{NTest}
N(b,T):=\#\{ \rho=\beta+i\gamma~:~ \ze(\rho)=0, \,\beta\geq b,
|\gamma|\leq T \}.
\end{equation}
The main result of the paper is the following classical style density estimate for the distribution of the zeros of the Beurling zeta function.
\begin{theorem}\label{th:density} Assume that $\G$ satisfies besides
Axiom A also Conditions B and G, too. Then for any $\varepsilon>0$
there exists a constant $C=C(\varepsilon,\G)$ such that for all sufficiently large $T$ and
$\alpha>(1+\theta)/2$ we have
\begin{equation}\label{density}
N(\alpha,T)\leq C T^{\frac{6-2\theta}{1-\theta}(1-\alpha)+\ve}.
\end{equation}
\end{theorem}

Note that according to the quite standard Lemma \ref{l:Littlewood} below, $N(\alpha,T)=O(T^{1+\ve})$ for $\alpha>\theta$, always. Thus the statement is nontrivial only if $\alpha$ is close to $1$, more precisely when $\alpha> \frac{5-\theta}{6-2\theta}$.

For the formulation of further two results see Sections \ref{sec:localdensity} and \ref{sec:cluster}. These additional results rely solely on Axiom A, and do not use the additional assumptions of Conditions B and G. However, the density estimate above can still be considered the main result of the paper in view of the wide range of prospective applicability in further studies of number theory nature.

The structure of the paper is as follows. In \cite{Rev-MP}, considered Part I of the series, we described a number of auxiliary results, estimations and formulae for the Beurling zeta function. For the convenience of the reader those also in use here will be presented without proofs in the next section \S \ref{sec:basics}. In \S \ref{sec:density} we prove Theorem \ref{th:density}. In \S \ref{sec:localdensity} the reader finds a generalized (and at the same time somewhat sharpened) version of the Turán type local density estimate mentioned above. Finally, in \S \ref{sec:cluster} we prove a streamlined and generalized version of the above mentioned zero clustering result of Ramachandra.

\section{Lemmata on the Beurling $\zeta$ function}\label{sec:basics}

\subsection{Basic properties of the Beurling $\zeta$.}
The following basic lemmas are just slightly more explicit forms of 4.2.6.
Proposition, 4.2.8. Proposition and 4.2.10. Corollary of
\cite{Knopf}. In \cite{Rev-MP} we elaborated on their proofs for explicit
handling of the arising constants in these estimates.
\begin{lemma}\label{l:zetaxs} Denote the "partial sums" (partial
Laplace transforms) of $\N|_{[1,X]}$ as $\ze_X$ for arbitrary $X\geq 1$:
\begin{equation}\label{zexdef}
\ze_X(s):=\int_1^X x^{-s} d\N(x).
\end{equation}
Then $\ze_X(s)$ is an entire function and for $\sigma:=\Re s
>\theta$ it admits
\begin{equation}\label{zxrewritten}
\ze_X(s)=
\begin{cases} \ze(s)-\frac{\kappa X^{1-s}}{s-1}-\int_X^\infty x^{-s}d\R(x)
& \textrm{for all} ~~~s\ne 1,  \\
\frac{\kappa }{s-1}-\frac{\kappa X^{1-s}}{s-1}+\int_1^X
x^{-s}d\R(x)
& \textrm{for all} ~~~s\ne 1, \\
\kappa \log X + \int_1^X \frac{d\R(x)}{x} & \textrm{for}~~~ s=1,
\end{cases}
\end{equation}
together with the estimate
\begin{equation}\label{zxesti}
\left|\ze_X(s) \right| \leq \ze_X(\sigma) \leq
\begin{cases} \min \left( \frac{\kappa X^{1-\sigma}}{1-\sigma}
+ \frac{A}{\sigma-\theta},~ \kappa X^{1-\sigma}\log X +
\frac{A}{\sigma-\theta}\right) &\textrm{if} \quad
\theta<\sigma < 1,
\\ \kappa \log X + \frac{A}{1-\theta}& \textrm{if} \qquad \sigma =1,
\\ \min\left( \frac{\sigma (A+\kappa)}{\sigma-1},~{\kappa}\log X + \frac{\sigma A}{\sigma-\theta}
\right) &\textrm{if}\quad
\sigma>1.
\end{cases}
\end{equation}
Moreover, the above remainder terms can be bounded as follows.
\begin{equation}\label{zxrlarge}
\left| \int_X^\infty x^{-s}d\R(x) \right| \leq A
\frac{|s|+\sigma-\theta}{\sigma-\theta} X^{\theta-\sigma},
\end{equation}
and
\begin{equation}\label{zxrlow}
\left| \int_1^X x^{-s}d\R(x) \right| \leq A \left(
|s|\frac{1-X^{\theta-\sigma}}{\sigma-\theta} +
X^{\theta-\sigma}\right) \leq A \min \left(
\frac{|s|}{\sigma-\theta},~|s| \log X + X^{\theta-\sigma} \right).
\end{equation}
\end{lemma}
\begin{lemma}\label{l:zkiss} We have
\begin{equation}\label{zsgeneral}
\left|\zs-\frac{\kappa}{s-1}\right|\leq \frac{A|s|}{\sigma-\theta}
\qquad \qquad\qquad (\theta <\sigma ,~ t\in\RR,~~ s\ne 1).
\end{equation}
In particular, for large enough values of $t$ it holds
\begin{equation}\label{zsgenlarget}
\left|\zs \right|\leq \sqrt{2} \frac{(A+\kappa)|t|}{\sigma-\theta}
\qquad \qquad\qquad (\theta <\sigma \leq |t|),
\end{equation}
while for small values of $t$ we have
\begin{equation}\label{zssmin1}
|\zs(s-1)-\kappa|\leq \frac{A|s||s-1|}{\sigma-\theta} \leq
\frac{100 A}{\sigma-\theta}\qquad (\theta <\sigma \leq 4,~|t|\leq
9).
\end{equation}
As a consequence, we also have
\begin{equation}\label{polenozero}
\zs\ne 0 \qquad \textrm{for}\qquad  |s-1| \leq
\frac{\kappa(1-\theta)}{A+\kappa}.
\end{equation}
\end{lemma}

\begin{lemma}\label{l:oneperzeta} We have
\begin{equation}\label{zsintheright}
|\zs| \leq \frac{(A+\kappa)\sigma}{\sigma-1} \qquad (\sigma >1),
\end{equation}
and also
\begin{equation}\label{reciprok}
|\zs| \geq \frac{1}{\ze(\sigma)} >
\frac{\sigma-1}{(A+\kappa)\sigma} \qquad (\sigma >1).
\end{equation}
\end{lemma}

\subsection{Estimates for the number of zeros of  $\zeta$}\label{sec:zeros}

\begin{lemma}\label{l:Jensen} Let $\theta<b<1$ and consider
the rectangle $H:=[b,1]\times [i(T-h),i(T+h)]$, where $h:=\frac{\sqrt{7}}{3}
\sqrt{(b-\theta)(1-\theta)}$ and $|T| \ge e^{5/4}+\sqrt{3}\approx 5.222...$ is arbitrary.

Then the number $n(H)$ of zeta-zeros in the rectangle $H$ satisfy
\begin{align}\label{zerosinH}
n(H) & \leq \frac{1-\theta}{b-\theta}\left(0.654 \log|T| + \log\log |T| + 6\log(A+\kappa) + 6\log\frac1{1-\theta} +12.5 \right)
\notag \\ &\leq \frac{1-\theta}{b-\theta}\left(\log|T| + 6\log(A+\kappa) + 6\log\frac1{1-\theta} +12.5 \right).
\end{align}

Moreover, if $|T|\le 5.23$, then we have analogously the $\log|T|$-free estimate
\begin{equation}\label{zerosinH-smallt}
n(H) \leq \frac{1-\theta}{b-\theta}\left(6\log(A+\kappa) + 6\log\frac1{1-\theta}+14\right).
\end{equation}
\end{lemma}
\begin{remark}\label{r:alsoindisk} In fact, this estimate includes also the total number $N$ of zeroes in the disc $\mathcal{D}_r:=\{s~:~|s-(p+iT)|\leq r:=p-q\}$, where $p:=1+(1-\theta)$ and
$q:=\theta+\frac23(b-\theta)$ are parameters introduced in its proof, see \cite{Rev-MP}.
\end{remark}

%
%

\begin{lemma}\label{l:Littlewood} Let $\theta<b<1$ and consider
any height $T\geq 5$ together with the rectangle $Q:=Q(b;T):=\{
z\in\CC~:~ \Re z\in [b,1],~\Im z\in [-T,T]\}$. Then the number of
zeta-zeros $N(b,T)$ in the rectangle $Q$ satisfy
\begin{equation}\label{zerosinth-corr}
N(b,T)\le \frac{1}{b-\theta}
\left\{\frac{1}{2} T \log T + \left(2 \log(A+\kappa) + \log\frac{1}{b-\theta} + 3 \right)T\right\}.
\end{equation}
\end{lemma}

\begin{lemma}\label{c:zerosinrange}
Let $\theta<b<1$ and consider any heights $T>R\geq 5 $ together
with the rectangle $Q:=Q(b;R,T):=\{ z\in\CC~:~ \Re z\in [b,1],~\Im
z\in (R,T]\}$.

Then the number of zeta-zeros $N(b;R,T)$ in the rectangle $Q$ satisfies\footnote{Here and below in \eqref{zerosbetweenone} a corrected formulation is presented. In the original calculation of \cite{Rev-MP}, when collecting terms in the end of the proof of (3.15), the term $8\pi\log T$ heading the third line of the long displayed formula occupying lines 13-17 of page 220 of [37], was erroneously neglected in the final count of the last line of this displayed sequence of inequalities. Following up the consequent changes, the resulting concrete estimations are the ones written here.}
\begin{equation}\label{zerosbetween}
N(b;R,T) \leq\frac{1}{b-\theta} \left\{ \frac{4}{3\pi} (T-R) \left(\log\left(\frac{11.4 (A+\kappa)^2}{b-\theta}{T}\right)\right)  + \frac{16}{3}  \log\left(\frac{60 (A+\kappa)^2}{b-\theta}{T}\right)\right\}.
\end{equation}

In particular, for the zeroes between $T-1$ and $T+1$ we have for
$T\geq 6$
\begin{align}\label{zerosbetweenone}
N(b;T-1,T+1) \leq \frac{1}{(b-\theta)} \left\{6.2 \log T +
6.2 \log\left( \frac{(A+\kappa)^2}{b-\theta}\right) + 24 \right\}.
\end{align}
\end{lemma}

\subsection{The logarithmic derivative of the Beurling $\zeta$}
\label{sec:logder}

\begin{lemma}\label{l:borcar} Let $z=a+it_0$ with $|t_0| \geq e^{5/4}+\sqrt{3}=5.222\ldots$ and $\theta<a\leq 1$. With $\delta:=(a-\theta)/3$ denote by $S$ the
(multi)set of the $\ze$-zeroes (listed according to multiplicity)
not farther from $z$ than $\delta$. Then we have
\begin{align}\label{zlogprime}
\left|\frac{\ze'}{\ze}(z)-\sum_{\rho\in S} \frac{1}{z-\rho}
\right| & < \frac{9(1-\theta)}{(a-\theta)^2}
\left(22.5+14\log(A+\kappa)+14\log \frac{1}{a-\theta} + 5\log |t_0|\right).
\end{align}

Furthermore, for $0 \le |t_0| \le 5.23$ an analogous estimate (without any term containing $\log |t_0|$) holds true:
\begin{equation}\label{zlogprime-tsmall}
\left|\frac{\ze'}{\ze}(z)+\frac{1}{z-1}-\sum_{\rho\in S} \frac{1}{z-\rho}
\right|  \le \frac{9(1-\theta)}{(a-\theta)^2}
\left(34+14\log(A+\kappa)+18\log \frac{1}{a-\theta}\right).
\end{equation}
\end{lemma}

\begin{lemma}\label{l:path-translates} For any given parameter
$\theta<b<1$, and for any finite and symmetric to zero set $\A\subset[-iB,iB]$ of cardinality $\#\A=n$, there exists a broken line $\Gamma=\Gamma_b^{\A}$, symmetric to the real axis and consisting of horizontal and vertical line segments only, so that its upper half is
$$
\Gamma_{+}= \bigcup_{k=1}^{\infty}
\{[\sigma_{k-1}+it_{k-1},\sigma_{k-1}+it_{k}] \cup
[\sigma_{k-1}+it_{k},\sigma_{k}+it_{k}]\},
$$
with $\sigma_j\in [\frac{b+\theta}{2},b]$, ($j\in\NN$),  $t_0=0$, $t_1\in[4,5]$ and $t_j\in
[t_{j-1}+1,t_{j-1}+2]$ $(j\geq 2)$ and satisfying that the
distance of any $\A$-translate $\rho+ i\alpha ~(i\alpha\in\A)$ of a $\zeta$-zero $\rho$ from any point $s=\sigma+it \in \Gamma$ is at least $d:=d(t):=d(b,\theta,n,B;t)$ with\footnote{As is mentioned in the footnote there, in Lemma \ref{c:zerosinrange} a slight correction of the formulation was due, entailing some corresponding corrections also here and in the next lemma. This involves only the values of the numerical constants, which--as long as they are some effective constants anyway--bear no further importance for us here.}
\begin{equation}\label{ddist-corr}
d(t):=\frac{(b-\theta)^2}{4n \left(12 \log(|t|+B+5) + 51 \log (A+\kappa) + 31 \log\frac{1}{b-\theta}+ 113\right)}.
\end{equation}
Moreover, the same separation from translates of $\zeta$-zeros holds also for the
whole horizontal line segments $H_k:=[\frac{b+\theta}{2}+it_k,2+it_k]$, $k=1,\dots,\infty$, and their reflections $\overline{H_k}:=[\frac{b+\theta}{2}-it_k,2-it_k]$, $k=1,\dots,\infty$, and furthermore the same separation holds from the translated singularity points $1+i\al$ of $\zeta$, too.
\end{lemma}

\begin{lemma}\label{l:zzpongamma-c} For any $0<\theta<b<1$ and symmetric to $\RR$ translation set $\A\subset [-iB,iB]$, on the broken line $\Gamma=\Gamma_b^{\A}$, constructed in the above Lemma \ref{l:path-translates}, as well as on the horizontal line segments $H_k:=[a+it_k,2+it_k]$ and  $\overline{H_k}$, $k=1,\dots,\infty$ with $a:=\frac{b+\theta}{2}$, we have uniformly for all $\alpha \in \A$
\begin{equation}\label{linezest-c}
\left| \frac{\ze'}{\ze}(s+i\alpha) \right| \le n
\frac{1-\theta}{(b-\theta)^{3}} \left(10 \log(|t|+B+5)+60\log(A+\kappa) + 42 \log\frac1{b-\theta}+ 140\right)^2.
\end{equation}
\end{lemma}

\subsection{Riemann-von Mangoldt type formulae of prime distribution
with zeroes of the Beurling $\zeta$}\label{sec:sumrho}

We denote the set of $\zeta$-zeroes, lying to the right of $\Gamma$, by $\Z(\Gamma)$, and denote $\Z(\Gamma,T)$ the set of those zeroes $\rho=\beta+i\gamma\in \Z(\Gamma)$ which satisfy $|\gamma|\leq T$. The next statement is Theorem 5.1 from \cite{Rev-MP}.

\begin{lemma}[Riemann--von Mangoldt formula]\label{l:vonMangoldt}
Let $\theta<b<1$ and $\Gamma=\Gamma_b^{\{0\}}$ be the curve defined in Lemma \ref{l:path-translates} for the one-element set $\A:=\{0\}$ with $t_k$ denoting the corresponding set of abscissae in the construction.
Then for any $k=1,2,\ldots$, and $4 \leq t_k<x$ we have
$$
\psi(x)=x - \sum_{\rho \in \Z(\Gamma,t_k)}
\frac{x^{\rho}}{\rho} + O\left( \frac{1-\theta}{(b-\theta)^{3}}
\left(A+\kappa+\log \frac{x}{b-\theta}\right)^3 \big( \frac{x}{t_k} + x^b \big)\right).
$$
\end{lemma}
Here let us call attention to the regrettable fact that in \cite{Rev-MP} a dumb error occurred in the last line of the proof of Theorem 5.1 (which is in fact the very last line of the whole paper, too). Namely, $x/t_k$ was erroneously estimated by $4x^b$. The formulation here is thus also a correction\footnote{Note that in formula (5.16) of \cite{Rev-MP} $\log(t_k)$ is superfluous in the right hand side, because it was already incorporated into the preceding log power.} of the mistake made in the formulation of Theorem 5.1 in \cite{Rev-MP} where instead of the correct last factor $\big(\frac{x}{t_k} + x^b\big)$, only $x^b$ was put down.

\section{A density theorem for $\ze$-zeros close to the $1$-line}
\label{sec:density}

\subsection{A discussion of condition G}
Recall that we have introduced Condition B and Condition G for use in this section.

There are many natural examples where Condition G is met. E.g. if
$\G$ is the ideal ring of an algebraic number field $\KK$, then a
well-known result, see e.g. Lemma 4.9. on p. 143 of \cite{Nar},
provides the estimate $G(m)\leq d_n(m)= O(m^{\varepsilon})$ for
all $\ve>0$ (where $d_n(m)$ is the classical $n$-term divisor
function and $n$ is the degree of the algebraic number field $\KK$
in question). It is clear that in case $G(\nu)=O(\nu^{\ve})$ also
Condition G must hold. Actually, $d_n(m)\leq d^{n-1}(m)$ and by
well-known number theory we also have
$$
\sum_{m<X} d^q(m) \sim C_q X \log^{{2^q}-1} X \qquad (X\to
\infty),
$$
so that Condition G holds for all exponent $p$.

More generally, let $\mathfrak A$ denote the category of all
finite abelian groups,  $\mathfrak S$ be the category of all
semisimple associative rings of finite cardinal, and $\mathfrak F$
be the category of all finitely generated torsion modules over the
ring $D$ of all algebraic integers in some given algebraic number
field $\KK$. One can also consider ${\mathfrak F}^{(\langle k
\rangle)}$, for any given finite or infinite sequence $\langle k
\rangle=k_1,\dots,k_n,\dots$, the category of all modules $M\in
\mathfrak F$ such that every indecomposable direct summand of $M$
is isomorphic to a cyclic module of the form $D/P^r$, where $P$ is
a prime ideal in $D$ and $r=k_i$ for some $i$, see \cite[p.
117]{Knopf}. These structures contain, as sub-semigroups, many
other important arithmetical categories like semisimple finite
dimensional associative algebras over a given field, certain
Galois fields, etc.: see \cite[Ch. 1]{Knopf} for details, in
particular page 16-21 for more detailed description of these and
related structures.

For abelian groups of finite order, the counting function
$G_{\mathfrak A}(m)$ has, by \cite[V.1.10 Corollary]{Knopf}, an
asymptotical $k^{\rm th}$ moment for every $k\in\NN$, which is a
strong form of the above Condition G: the $o(\log X)$ function is
just $C+o(1)$. In $\mathfrak S$ the counting function
$G_{\mathfrak S}(n)$ also has, by \cite[V.1.13 Theorem]{Knopf}, an
asymptotical $k^{\rm th}$ moment for every $k\in\NN$, implying
again the strong form of Condition G. For $\mathfrak F$ this is
found in \cite[V.1.9. Theorem]{Knopf}, too. (The reason of this is
the intimate relationship of the value of $G$ on prime powers
$p^r$ with the partition function $p(r)$, see p. 124 of
\cite{Knopf}).

Note that these categories (and many others) all
satisfy Axiom A, see \cite{Knopf}, p. 16-20 and 120-121, e.g.

Let us recall the following facts from the theory of arithmetical
semigroups (see \cite{Knopf}, IV.4.1. Proposition and V.2.9.
Theorem.)

\begin{lemma}[Knopfmacher]\label{l:dksum} Let $\G$ be an
arithmetical semigroup satisfying Axiom A. Then for the divisor
function $d(g)$ on $\G$ we have with any $k\in \NN$ the asymptotic
equivalence formula
\begin{equation}\label{dksum}
\sum_{g\in\G;|g|<X} d^k(g) \sim  A_0
X \log^{2^k-1} X \qquad (X\to \infty),
\end{equation}
with $A_0=A_0(\GG,k)$ a nonzero constant. We also have
\begin{equation}\label{dlittle}
\limsup_{|g|\to \infty} \frac{\log d(g) \log\log |g|}{\log |g|} =
2,
\end{equation}
so $\log d(g)=o(\log|g|)$, too.
\end{lemma}

Let us note that Knopfmacher uses the notation $B_k$ for the constant $A_0(\GG,k)$ occurring in formula \eqref{dksum}, and gives the semi-explicit expression $A_0=\frac{\kappa^{2^k}}{(2^k-1)!} \widetilde{g_k}(1)$ with $\widetilde{g_k}$ defined through the process of the proof of IV.3.7. Proposition of \cite{Knopf}. For a further discussion on the value of the constant see \cite{Rev-MP}.

\subsection{Proof of Theorem \ref{th:density}}

\begin{proof} We apply the by now standard treatment of
zero-detecting sums and large sieve type estimates. The derivation
here follows the relatively simple, straightforward argument in
\cite{Pintz9}, see also \cite{PintzNewDens}.

Let $Y$ be a large parameter, so that $T^3<Y<T^{\frac{3}{1-\theta}}$. We define the
arithmetical function $a(g):=a_T(g)$ on $\G$ as
$$
a(g):=\sum_{h|g;~|h|\leq T} \mu(h).
$$
Clearly, $a(1)=1$ and for $1<|g|\leq T$ we have $a(g)=0$: for
other values of $g$ we have $|a(g)|\leq d(g)$. For any complex
number $z\in\CC$ put
\begin{equation}\label{Hdef}
H:=H(z):=H(T,Y,z):=\sum_{|g|<Y}
\frac{a(g)}{|g|^z}=1+\sum_{T<|g|<Y} \frac{a(g)}{|g|^z}
\end{equation}
By Lemma \ref{l:Littlewood}, \eqref{zerosinth-corr} we already know about
the number of zeroes that $N(\alpha,X) \ll
\frac1{1-\theta}X \log\frac{X}{1-\theta}$, so using also
$\alpha>\frac{1+\theta}{2}$ and choosing $X:=\log^2 T
~Y^{1-\alpha}$ the number of zeroes below the height $X$ can be
estimated as
\begin{align}\label{NX}
N(\alpha,X) \ll \left(\frac{1}{1-\theta} \log T\right)^3 \,Y^{1-\alpha}.
\end{align}
On the other hand for the zeroes counted in
$N(\alpha,T)-N(\alpha,X)$, we select a separated "one-covering",
i.e. we take (one) zero $\rho_1=\beta_1+i\gamma_1$ with
$\beta_1>\alpha$, $\gamma_1\geq X$ and of minimal $\gamma_1$, and
then inductively, after $\rho_j$ has already been selected, we
take the next $\rho_{j+1}=\beta_{j+1}+i\gamma_{j+1}$ with minimal
$\gamma_{j+1}\geq \gamma_j+1$ but remaining $\leq T$. Clearly
$\gamma_j +1 \leq \gamma_{j+1}$ entails that the set
$\Z:=\{\rho_j\}$ of all the zeroes selected is finite: let its
number of elements be $Z:=\# \Z$, say.

By construction, the imaginary part of any $\ze$-zero in the
rectangle $[\alpha,1]\times[X,T]$ is within 1 to some element
$\rho_j$ of the set $\Z$, so by Lemma \ref{c:zerosinrange} \eqref{zerosbetweenone} and taking
into account $\alpha>\frac{1+\theta}{2}$ and the symmetry of the
zeroes with respect to the real line, too, we obtain with some implied explicit constant
$A_2:=A_2(\theta,A,\kappa)$ and for sufficiently large $T$ and $X$ the estimate
\begin{equation}\label{NXTZ}
N(\al,T)-N(\al,X) \ll Z \log T.
\end{equation}
Finally, we are to estimate $\# \Z =Z$. Firstly, for any
$\rho=\rho_j\in \Z$, we can write
\begin{align*}
H:=H(T,Y,\rho_j)& =\sum_{|g|\leq Y}
\frac1{|g|^{\rho}}\sum_{h|g~|h|\leq T} \mu(h) = \sum_{|h|\leq T}
\frac{\mu(h)}{|h|^{\rho}} \sum_{f ~ |f|\le Y/|h|} \frac1{|f|^\rho},
\\ |H| & \leq \sum_{|h|\leq T} \frac{1}{|h|^{\beta}}
\left|\ze_{Y/|h|}(\rho) \right|.
\end{align*}
Here using $\ze(\rho)=0$, too, the inner expression can be
estimated by Lemma \ref{l:zetaxs}, \eqref{zxrewritten} (first
line) and \eqref{zxrlarge}. We gain
$$
\left|\ze_{Y/|h|}(\rho) \right| \leq
\frac{\kappa\left(\frac{Y}{|h|}\right)^{1-\beta}}{|\rho-1|}+
A\left(\frac{Y}{|h|}\right)^{\theta-\beta}\left(
\frac{|\rho|}{\beta-\theta}+1 \right) \leq \frac{\kappa}{\gamma}
\left(\frac{Y}{|h|}\right)^{1-\beta}+
\frac{4A\gamma}{1-\theta} \left(\frac{Y}{|h|}\right)^{\theta-\beta}.
$$
Applying this in the above estimation of $H$ we are led to
$$
|H|\leq \frac{\kappa Y^{1-\beta}}{\gamma} \sum_{|h|\leq T}
\frac{1}{|h|} + \frac{4A\gamma}{1-\theta}
Y^{\theta-\beta} \sum_{|h|\leq T} \frac{1}{|h|^{\theta}}=
\frac{\kappa Y^{1-\beta}}{\gamma} \ze_T(1) +
\frac{4A\gamma}{1-\theta} \gamma Y^{\theta-\beta}
\ze_T(\theta).
$$
From the second line of \eqref{zxrewritten} and the second part of \eqref{zxrlow} in Lemma
\ref{l:zetaxs} we get
$$
|\ze_T(\theta)|\leq \frac{\kappa T^{1-\theta}}{1-\theta}
+\left|\int_1^T x^{-\theta} d\R(x) \right| \leq \frac{\kappa
T^{1-\theta}}{1-\theta} + A + A \theta\log T \le \frac{A+\kappa}{1-\theta} T^{1-\theta} \log T,
$$
using in the end $T>\exp(\frac{1}{1-\theta})$; and from Lemma \ref{l:zetaxs}, \eqref{zxesti}, second line
we obtain$$
\ze_T(1) \leq \frac{A+\kappa}{1-\theta}\log T.
$$
Inserting these into the last estimation of $H$ leads to
\begin{align*}
|H| & \leq \frac{\kappa(A+\kappa)\log T\, Y^{1-\beta}}{(1-\theta)\gamma} +
\left(\frac{4A(A+\kappa)}{(1-\theta)^2}\right) \gamma \log T \,T^{1-\theta}
Y^{\theta-\beta}\\ & \leq (4A+\kappa)(A+\kappa)\log T \left(
\frac{Y^{1-\beta}}{(1-\theta)\gamma} + \frac{\gamma
Y^{\theta-\beta}T^{1-\theta}} {(1-\theta)^2}\right).
\end{align*}
Assuming now that
\begin{equation}
Y\geq Y_0(T,\theta):=\left(\frac{1}{1-\theta} T^{3-\theta}\right)^{\frac1{1-\theta}},
\end{equation}
the second term is always below the first in view of $|\gamma|\leq
T$. Furthermore, $X \leq \gamma$, so writing in $X:=\log^2 T~
Y^{1-\alpha}$ in place of $\gamma$ we infer
$$
H\leq 8(A+\kappa)^2 \log T \frac{Y^{1-\beta}}{(1-\theta)\gamma} \leq
\frac{8(A+\kappa)^2}{(1-\theta)\log T} Y^{\alpha-\beta} \leq
\frac{8(A+\kappa)^2}{(1-\theta)\log T} < \frac12,
$$
if $T>\exp\left(\frac{16(A+\kappa)^2}{1-\theta}\right)$. Therefore, in \eqref{Hdef} the first
constant $1$ is dominating, and via $|H-1|>1/2$ we obtain
\begin{align}\label{Ysumdef}
\frac14 \cdot Z & \leq \sum_{j=1}^Z |H(\rho_j)-1|^2 = \sum_{j=1}^Z
\left| \sum_{T<|g|<Y} \frac{a(g)}{|g|^{\rho_j}}\right|^2 \\
&\leq \sum_{j=1}^Z \left\{ \sum_{k=0}^{[\log (Y/T)]} \left|
\sum_{e^k T <m \leq \min\{e^{k+1}T, Y\}} \frac{\left(
\sum_{g~:~|g|=m} a(g)\right)}{m^{\rho_j}}\right|\right\}^2 \notag\\
&\leq \left(\log \left(\frac YT\right) +1\right) \sum_{k=0}^{[\log (Y/T)]} \sum_{j=1}^Z \left|
\sum_{M_k<m\leq N_k} \frac{\left( \sum_{g~:~|g|=m}
a(g)\right)}{m^{\rho_j}}\right|^2,\notag
\end{align}
with some appropriate $T\leq M_k < N_k < e M_k$, $N_k\leq Y$. First we
want to estimate the coefficients of the Dirichlet series, so let
us write
\begin{equation}\label{Fm}
F(m):=\sum_{g~:~|g|=m} a(g)
\end{equation}
and use Cauchy's inequality and the trivial upper estimate
$|a(g)|\leq d(g)$ to obtain
\begin{align}\label{Fmestimate}
|F(m)|^2 & \leq \sum_{g~:~|g|=m} 1 \cdot \sum_{g~:~|g|=m} a^2(g) =
G(m) \sum_{g~:~|g|=m} a^2(g) \\ & \leq G(m) \sum_{g~:~|g|=m}
d^2(g) = \sum_{g~:~|g|=m} G(|g|)d^2(g). \notag
\end{align}
Recall that the exponents $\rho_j$ in the inner Dirichlet
polynomials of the last double sum in \eqref{Ysumdef} are all
counted in $N(\alpha,T)$, hence $\beta\geq \al$ and $|\gamma|\leq
T$. By the large sieve type inequality of \cite[Theorem
7.5]{Mont} and writing $M:=M_k, N:=N_k$ here, we are led to
\begin{align}\label{largesieve}
\sum_{j=1}^Z \left| \sum_{M<m\leq N}
\frac{F(m)}{m^{\rho_j}}\right|^2 & \ll (T+N) \log N \sum_{m=M}^N
\frac{F^2(m)}{m^{2\alpha}} \bigg( 1+ \log\frac{\log 2N}{\log 2m} \bigg)\notag \\
& \ll N \log N \sum_{m=M}^N \frac{F^2(m)}{N^{2\al}} = \log N ~
N^{1-2\al} \sum_{n=1}^N F^2(m).
\end{align}
Here we use \eqref{Fmestimate} and apply H\"older's inequality
\emph{while summing over elements of} $\G$ with some exponent $p$
satisfying Condition G. This yields with $q:=\frac{p}{p-1}$
\begin{equation}\label{Fest}
\sum_{n=1}^N F^2(m) \leq \sum_{g~:~|g|\leq N} G(|g|)d^2(g) \leq
\left( \sum_{g~:~|g|\leq N} G(|g|)^p \right)^{1/p} \left(
\sum_{g~:~|g|\leq N} d^{2q}(g) \right)^{1/q}.
\end{equation}
By Lemma \ref{l:dksum}, the second sum is $O(N \log^{C(q)} N)$,
while the first sum is by Condition G $\ll N^{1+\ve}$. Collecting
\eqref{Ysumdef}, \eqref{largesieve} and \eqref{Fest} leads to
\begin{equation}\label{Densityfinal}
Z \ll \sum_{k=0}^{[\log (Y/T)]} \log Y \log N_k~ N_k^{1-2\al} N_k^{(1+\varepsilon)/p}
\left( N_k \log^{C(q)} N_k \right)^{1/q}  \ll Y^{2-2\al+\ve}.
\end{equation}
At last, we can choose $Y$ the smallest possible, that is
$Y:=Y_0$, to get
$$
Z \ll Y_0^{2-2\al+\ve}= \left( \frac{1}{(1-\theta)^{\frac{1}{1-\theta}}}
T^{\frac{3-\theta}{1-\theta}} \right)^{2-2\al+\ve}\ll
T^{\frac{6-2\theta}{1-\theta}\left\{(1-\al)+\ve\right\}}.
$$
To finish the proof we need only to combine this estimate with
\eqref{NX} and \eqref{NXTZ}.
\end{proof}

The result thus shows that, e.g., the functional equation, so
fundamentally present in several approaches, is not necessary for
a density theorem to hold. On the other hand positivity of the
coefficients of the Dirichlet series does play a role here. In
this respect Theorem \ref{th:density} is complement to the similar
density theorem of Kaczorowsky and Perelli \cite{KP}, where the
Selberg class of zeta functions are shown to admit such density
estimates. It would be interesting to analyze what essentially
minimal set of properties, assumed on the class of zeta functions
considered, can still imply the validity of such density estimates.

\section{A Tur\'an type local density theorem for $\ze$-zeros close
to the boundary of the zero-free region} \label{sec:localdensity}

For arbitrary $\tau>0$ and $\theta<\si<1$, we denote the rectangle
$$
Q_{\sigma,h}(\tau):=[\si,1]\times[i(\tau-h),i(\tau+h)].
$$

\begin{theorem}\label{th:locdens} Let $(1+\theta)/2 < b\leq 1$,
$2\leq h$ and $\tau>\max(2h,\tau_0)$ where
$\tau_0=\tau_0(\theta,A,\kappa)$ is a large constant depending on
the given parameters of $\zs$.

Assume that $\zs$ does not vanish in the rectangle $\sigma\geq b$,
$|t-\tau|\leq h$, denoted by $Q_{b,h}(\tau)$, i.e. that
$N(b,\tau-h,\tau+h)=0$. Then for any $\delta$ with
$15\frac{\log \log \log \tau}{\log \log \tau} < \de < \frac{b-\theta}{10}$ we have
$$
M:=N(b-\delta,\tau-\delta,\tau+\delta) \ll \de  \log\tau,
$$
with an implied absolute constant not depending on $\GG$.
\end{theorem}

The original work of Tur\'an, aiming at
"almost getting the density hypothesis", involves the condition
$h\gg \log\tau$. The present analysis reveals that this is not
necessary: we get the same result with any $h>2$ as well. Thus the
result - at least formally - gives something new even for the
Riemann zeta function.

In proving the result we follow closely the original work of Tur\'an,
cf. the Appendix of \cite{TuranD}.

\begin{lemma}\label{l:nonvanishsmall} Let the parameters $b,h,\tau$ be fixed as in the previous theorem, and assume that $\tau$ is large enough. If $N(b,\tau-h,\tau+h)=0$ and $\lambda:= 5 \log\log \log \tau /\log\log \tau$, then for all $s=\si+it$ with $b+3\lambda\le \si\le 1.7$ and $|t-\tau|\le h-1$ the inequality
\begin{equation}\label{zetainnonvanish}
\left|\frac{\ze'}{\ze}(s) \right| \leq \frac{\log \tau}{(\log\log \tau)^2}
\end{equation}
holds true.
\end{lemma}
\begin{proof} Let $g(s):=\log\left(\zs/\ze(p+i\omega)\right)$, where $\tau-h+1\le \omega \le \tau+h-1$
and $1<p < 1+b$ is to be determined later. This function is analytic in
the disk $D:=\{s~:~|s-(p+i\omega)|\leq R:=p-b<1\}$, as $\zs$ does not vanish (for
$|t-\tau| \le |\omega-\tau|+R \leq h-1+R \leq h$ and $\Re s \ge p-R=b$).
Furthermore, the real axis, (whence the pole of $\zs$ at $s=1$) is at least as far from any point $s\in D$ as $\omega-R>\tau-h+1-R>\tau-h>h$. Whence $g(s)\ne 0,\infty$ in $D$ and $\log g(s)$ is analytic in $D$.

Clearly, in the center of this disk $g(p+i\omega)=0$. For the upper
estimation of the real part of $g(s)$ along the circle, let us use
\eqref{zsgenlarget} and \eqref{reciprok}, which leads to
\begin{align}\label{Regsest}
\Re g(s) & \leq
\log\left(\frac{\sqrt{2}(A+\kappa)}{b-\theta}\right) + \log(\omega+1) +
\log\left(\frac{p(A+\kappa)}{p-1}\right)
\log\left(\frac{5p(A+\kappa)^2}{(1-\theta)(p-1)}\right),
\end{align}
using also that $b-\theta > \frac{1}{2}(1-\theta)$ and also that $\omega>2$, say.

We will chose
$$
p:=\frac32,\qquad \Delta:=k \frac{\log \log \log \omega}{\log\log \omega}
<0.1, \quad (4 \le k \in\NN).
$$
So let us assume that $\tau$ is large enough, say $\tau > \tau_1:=
30\frac{(A+\kappa)^2}{1-\theta}$: then by $\omega>\tau/2$ in the last
estimate of \eqref{Regsest} the $\log \omega$ term dominates, and $\Re g(s) \leq 2\log \omega$.
Let $R':=R-\Delta < R $: by the Borel-Carathéodory theorem,
$$
\max_{|s-(p+i\omega)|\leq R'} |g(s)| \leq \frac{2R}{\Delta} \, 2\log \omega
< \frac4{\Delta} \log \omega.
$$
Next we apply the three-circle theorem to the circles
$C_1:=\{s~:~|s-(p+i\omega)|\leq R'\}$, $C_2:=\{s~:~|s-(p+i\omega)|\leq R"\}$
with $R":=R-2\Delta$, and $C_3:=\{s~:~|s-(p+i\omega)|\leq q\}$, where
now $q:=(p-1)/2= 1/4$, so on $C_3$ by the estimates of Lemma
\ref{l:oneperzeta} we have
$$
\max_{C_3} |g(s)| \leq \log \left( \frac{(A+\kappa)^2
(p-q)^2}{(p-q-1)^2} \right) = \log \left( 25 (A+\kappa)^2\right).
$$
From the three-circle theorem we obtain
\begin{align*}
\max_{|s-(p+i\omega)|\leq R"}&  \log |g(s)|  \leq \frac{\Delta}{R'} \log
\log \left( 25 (A+\kappa)^2\right) + \frac{R'-\Delta}{R'} \log
\left(\frac{4}{\Delta} \log \omega \right) \\ &= \frac{\Delta}{R'} \log
\left\{\frac{\Delta}{4}\log \left( 25 (A+\kappa)^2\right) \right\}
+\log \frac{4}{\Delta}+ \left(1- \frac{\Delta}{R'} \right)\log
\log \omega.
\end{align*}
We have $\omega>\tau-h+1> \tau/2$, whence for $\tau > \tau_2(A,\kappa)$ the first
expression is negative because $\Delta=\Delta(\omega)\to 0$ when $\omega\to \infty$, entailing that eventually also $(\Delta/4) \log(25(A+\kappa)^2)$ becomes smaller than 1.

Moreover, the fraction $\Delta/R'$ is at
least $\Delta$, since $R'<R<1$. Also, for $\tau > 2\exp(\exp(e))$ we have $\omega>\exp(\exp(e))$, too; in this case we also have $4/\Delta \le 4/k ~\log\log\omega /\log\log\log \omega \le \log \log \omega$.
Therefore, we end up with
\begin{align*}
\max_{|s-(p+i\omega)|\leq R"} \log |g(s)|  & \leq \log \frac{4}{\Delta}+
\left(1- {\Delta}\right)\log \log \omega \\ & \leq \log \log \log \omega + \log
\log \omega - k \log \log \log \omega.
\end{align*}
Choosing $k=4$ we conclude
$$
\max_{|s-(p+i\omega)|\leq R"} |g(s)|  \leq \frac{\log \omega}{(\log\log
\omega)^3}.
$$
Finally, we choose $r:=R-3\Delta=R"-\Delta$, and apply the
standard Cauchy estimate for the value of $|g'(s)|$ for any $s$ in the
disk $|s-(p+i\omega)|\leq r$. We thus obtain
$$
|g'(s)| \leq \frac1{\Delta}\frac{\log \omega}{(\log\log \omega)^3}
=\frac{\log \omega}{4\log\log\log \omega (\log\log \omega)^2} < \frac{\log \tau}{
(\log\log \tau)^2}.
$$
We have the same all over the area of these disks, and as $\tau$
is large enough, and $k=4$ we have $\lambda > \Delta$. Hence for
large enough $\tau$ the horizontal diameter $[b+3\Delta+ i\omega, 2p-b-3\Delta+i\omega]$
of the disk $|s-(p+i\omega)|\leq r$ covers $[b+3\lambda+i\omega,1.7+i\omega]$.
Using this for all choices of $\omega$ with $|\omega-\tau|\leq h-1$,
we obtain that the rectangle $[b+3\lambda,1.7]\times [\tau-h+1,\tau+h-1]i$
is contained in the total area covered by
the union of these disks.

But $g'(s)=(\log \zs)'=\zeta'(s)/{\zs}$, and so at any point $s=\sigma+it$,
$b+3\lambda\le \sigma \le 1.7$, $|t-\tau|\leq h-1$ we thus have
\eqref{zetainnonvanish}.
\end{proof}

\begin{proof}[Proof of Theorem \ref{th:locdens}]
Let $s=\sigma+it \in Q:=Q_{b+3\lambda,h-1}$ (with $\lambda$ as in Lemma
\ref{l:nonvanishsmall}) and consider the number of zeroes $N$ in
the disk $D(s,r)$ with $r:=(\sigma-\theta)/3$. As $\theta<\sigma<1$,
the disc $K:=\{z~:~|z-(1+(1-\theta)+it)| \leq
R:=2-2\theta-2(\sigma-\theta)/3\}$ contains $D(s,r)$, whence by Remark
\ref{r:alsoindisk} (applied to $\sigma$ in place of $b$) the number of
$\ze$-zeroes $N$ in $D(s,r)$ can at most be as large as
$$
N \leq \frac{1-\theta}{\sigma-\theta} \left(A_1+ \log t\right)\le \frac{1-\theta}{b-\theta} \left(A_1+ \log t\right) \le 2\left(A_1+ \log t\right),
$$
with some constant $A_1=A_1(\theta,A,\kappa)$.
Note that the same disk $D(s,r)$ occurs in Lemma \ref{l:borcar}:
hence we obtain now for the (multi)set $S$ of zeroes within
$D(s,r)$ that
\begin{equation}\label{rezlogprime}
\Re \left\{-\frac{\ze'}{\ze}(s)+\sum_{\rho\in S}
\frac{1}{s-\rho}\right\} \leq
\left|\frac{\ze'}{\ze}(s)-\sum_{\rho\in S} \frac{1}{s-\rho}
\right| \ll \left(A_1+\log t \right) \ll \log\tau,
\end{equation}
with $\tau$ (and whence $t$) large enough (e.g. for $\tau>2\exp(A_1)$).
The real parts of the expressions with the zeroes in the sum \eqref{rezlogprime} can
be rewritten for $\rho=\beta+i\gamma$ as
$$
\Re \frac{1}{s-\rho} = \frac{\sigma-\beta}{(\sigma-\beta)^2+(t-\gamma)^2},
$$
hence from \eqref{rezlogprime} we obtain with a suitably large
absolute constant $C$ and for $\tau$ (and hence $t$) large enough
\begin{equation}\label{resumma}
\sum_{\rho\in S}\frac{\sigma-\beta}{(\sigma-\beta)^2+(t-\gamma)^2} < \Re
\frac{\ze'}{\ze}(s) + C \log \tau .
\end{equation}
We choose now $s:=b+\delta+i\tau$, i.e. $\sigma=b+\delta$ and
$t:=\tau$. By assumption, we have $3\lambda<\delta$. That is,
$\sigma+it\in [b+3\lambda,1.7]\times [\tau-h+1,\tau+h-1]i$ and the
estimations of the previous Lemma \ref{l:nonvanishsmall} can be
applied to bound the arising $\ze'/\ze(\sigma+it)$ in the right hand
side. This yields
\begin{equation}\label{resummaplus}
\sum_{\rho\in S}\frac{\sigma-\beta}{(\sigma-\beta)^2+(\tau-\gamma)^2} < \frac{\log\tau}{(\log\log\tau)^2} + C \log \tau \le 2C \log \tau,
\end{equation}
the second term dominating.

It is clear that the terms in the sum are nonnegative, for we have assumed that $Q_{b,h}(\tau)$ is zero-free, the radius of the disk $D(s,r)$ is less than 1, and $|t-\tau|<h-1$. Therefore, any term can be dropped while preserving the inequality.

Observe that $Q_{b-\delta,\delta}(\tau)\cap S \subset [b-\de,b]\times [\tau-\de,\tau+\de]i \subset D(s,r)$, so that we can restrict the sum for summing over zeroes in $Q_{b-\de,\de}(\tau)$. Therefore,
$$
\sum_{\rho\in Q_{b-\delta,\delta}} \frac{\sigma-\beta}{(\sigma-\beta)^2+(\tau-\gamma)^2} < 2C \log \tau  .
$$
Now each term on the left hand side is at least
$2\delta/5 \delta^2$, hence we arrive at
$$
M \frac1{\delta} \le 5C\log \tau.
$$
The proof of the theorem concludes.
\end{proof}

\section{Clustering of zeroes in the vicinity of the
1-line}\label{sec:cluster}

In \cite{DMV}, Theorem 2 the authors prove that the zeroes
$\rho=\beta+i\gamma$ of the Beurling zeta function, close to
the one-line in the sense that $\beta>1-c/\sqrt{\log \gamma}$,
show a phenomenon of clustering: they do not occur in
isolation, but instead once a zero $\rho_0$ occurs, there must be further
ones in the union of some small discs $D(1+i\gamma_0,\lambda)$
and $D(1+i2\gamma_0,\lambda)$ around $1+i\gamma_0$ and
$1+i2\gamma_0$.

This theorem is itself a sharpening of what was proved in the average
by Montgomery in his monograph \cite{Mont} for the case of the Riemann
zeta function. Yet there is a less quoted, nevertheless sharper result,
due to Ramachandra \cite{Ram}, which provides similar clustering with closeness
relaxed to $1/\log \log \gamma$ only and still localizes to the small discs
$D(1+i\gamma_0,\lambda)$ and $D(1+i2\gamma_0,\lambda)$ around $1+i\gamma_0$
and $1+i2\gamma_0$.

It is worthy to work out the result here not only for sake of
generality but also to give a somewhat more transparent deduction
of the result. Indeed, Ramachandra uses a positive trigonometric
polynomial with obscure\footnote{It seems that the reasons lie in
the later, more general application of his method to clustering
around $1+i\gamma_1$ and $1+i\gamma_2$, when
two close zeroes are known at hight $\gamma_1$ and $\gamma_2$,
respectively.} coefficients (like $10^8$). Here we analyze the method and show
that the most common $3+4\cos\theta+\cos 2\theta$ does the job as
well.

Note that the further paper of Balasubramanian and Ramachandra \cite{BaRa},
claiming to achieve the even nicer localization of clustering right
in $D(1+i\gamma_0,\lambda)$, contains a fatal error\footnote{Indeed,
summarizing the previous Lemmas 3 and 5 in Lemma 6 on page 11, the
authors neglect a term $-2X^{1-s}\Gamma(\frac{1-s}{2})$, which is a
main term and destroys everything.}, unfortunately.

\begin{theorem}\label{th:Ramachandra} Assume $\zeta(\rho_0)=0$
with $\rho_0=\beta_0+i\gamma_0$, $\beta_0>\max(\theta,0.999)$, $\gamma_0\ge 100$, and
$1-\beta_0< \frac{1-\theta}{40\log\log\gamma_0}$.
Further, let the parameter $0<\lambda\le \frac23 (1-\theta)$ be arbitrary.

Then there exists an effective constant $A_{10}:=A_{10}(\theta,A,\kappa)$, depending only on the
parameters of the Beurling zeta function given in Axiom A, so that
with any value of the further parameter $Y$ satisfying $Y> \max\left(A_{10},\frac{4}{1-\theta}\log\log \gamma_0\right)$
we have
\begin{equation}\label{eq:Ramachandrastatement}
\sum_{\rho\in D(1+i\gamma_0,\lambda)\setminus \{\rho_0\}} e^{-Y(1-\beta)} +
\sum_{\rho\in D(1+i2\gamma_0,\lambda)} e^{-Y(1-\beta)}
\gg \frac{\lambda}{Y(1-\beta_0)}-c_0,
\end{equation}
with $c_0$, as well as the implied constant in $\gg$ being explicit absolute constants.
\end{theorem}

\begin{proof} We will work with the kernel function
$$
K(w,x):=x^w\Gamma\left(\frac{w}{2}\right),
$$
satisfying the integral formula
$$
d(x,\nu):=\frac{1}{2\pi i} \int\limits_{(2)} \frac{1}{\nu^w}
x^w\Gamma\left(\frac{w}{2}\right) dw = 2 \exp \left(
-\left(\frac{\nu}{x}\right)^2\right) \qquad \left( \nu\geq 1,
~\nu\in \RR\right),
$$
where, as usual, $\int_{(2)}=\int_{2-i\infty}^{2+i\infty}$. At the
outset we let $s\in \CC$ be arbitrary with $\Re s>1$. We put
$$
f(s):=f_{\gamma_0}^B(s) f_{2\gamma_0}^C(s)\ze^{-2D}(s), \qquad
\textrm{with}\quad
f_{\al}(s):=\zeta^2(s)\zeta(s-i\al)\zeta(s+i\al),
$$
so that
\begin{align*}
\frac{f'}{f}(s)&= B \frac{f'_{\gamma_0}}{f_{\gamma_0}}(s)
+C \frac{f'_{2\gamma_0}}{f_{2\gamma_0}}(s) - 2D\frac{\ze'}{\ze} (s) =
2(B+C-D) \frac{\ze'}{\ze}(s) \\
&\qquad\qquad+ B \left(\frac{\ze'}{\ze} (s-i\gamma_0)+ \frac{\ze'}{\ze}
(s+i\gamma_0)\right) + C \left(\frac{\ze'}{\ze} (s-2i\gamma_0)+
\frac{\ze'}{\ze} (s+i2\gamma_0)\right)
\\ &= - 2 \sum_{g\in \G}
\frac{\Lambda(g)}{|g|^s} \left( (B+C-D)+B \Re |g|^{i\gamma_0}+
C\Re |g|^{i2\gamma_0} \right)= - 2 \sum_{g\in \G}
\frac{\Lambda(g)P(\gamma_0 \log|g|)}{|g|^s}.
\end{align*}
with the choice of the constant parameters $B,C,D>0$ such that
\begin{equation}\label{Pdef}
P(u):= (B+C-D) + B \cos u + C \cos(2u) \geq 0 ~~(\forall u\in\RR).
\end{equation}
For any $s$ in the right halfplane $\Re s>1$ we have
$$
F(s,x):=- \sum_{g\in \G}
\frac{\Lambda(g)P(\gamma_0 \log |g|) d(x,|g|)}{|g|^s}=
\frac{1}{2\pi i} \int_{(2)} \frac12 \frac{f'}{f}(s+w) K(w,x) dw.
$$
Similarly to the von Mangoldt type formula of Lemma \ref{l:vonMangoldt}--whose
proof is detailed in \cite{Rev-MP}--we move the line of integration to
the left. For that, we fix a parameter $b$,
choose $\A:=\{-2\gamma_0, -\gamma_0, 0, \gamma_0, 2\gamma_0\}$, and consider the curve
$\Gamma:=\Gamma^{\A}_b$, as constructed in Lemma \ref{l:path-translates}. It is
appropriate here to fix our parameters $a, b$, frequently used below. We choose
$a:=\theta +(1-\theta)/3$, $b:=\theta+2(1-\theta)/3$, so that all
the constants $a-\theta$, $1-b$, $b-\theta$ are equivalent to $1-\theta$,
that is, they are all estimated both from below and from above by
a positive, finite constant multiple of $1-\theta$.
Correspondingly, in the following the constants $A_1, A_2, \ldots$
will be chosen to depend on $\theta, A$ and $\kappa$ only, but not on $a,b$.
Also note that with this choice
we have $a=\frac{b+\theta}{2}$, matching with the construction
of Lemma \ref{l:path-translates}.

Next, we move the part $[2-it_k,2+it_k]$ of the contour of integration to the left
to the corresponding part of $\Gamma-s$ along the horizontal
lines $t=\pm t_k$. Now it is easy to see that the "bridges"
along any horizontal segments $t=\pm t_k$ give $o(1)$ contribution as $t_k \to \infty$,
hence by letting $t_k\to\infty$ the whole line $\Re w=2$ of integration can
be moved to $\Gamma-s$ (so that $w+s \in \Gamma$).

Observe that for real $s$ by construction in Lemma \ref{l:path-translates} the broken line $\Gamma-s$
does not contain any singularities of the integrand, i.e. $\Gamma$ does not
contain singularities of $ f'/f $. In fact, we will only consider
parameter values $1<s<1+1/100$ in the proof. Note that for real $s$
we can also write $F(s,x)=\Re F(s,x)$ and
$$
\frac12 \Re \frac{f'}{f}(s) = \frac12 \frac{f'}{f}(s) = E \frac{\ze'}{\ze}(s)+ B \Re
\frac{\ze'}{\ze} (s-i\gamma_0) + C \Re \frac{\ze'}{\ze} (s-2i\gamma_0)
\quad(E:=B+C-D).
$$

We thus find, using again the notation $\Z(L)$ for the set of
$\zeta$-zeroes to the right of the curve $L$, that
\begin{align*}
F(s,x)&=\frac{1}{2\pi i} \int\limits_{\Gamma-s} \frac12
\frac{f'}{f}(s+w)K(w,x) dw + \frac12 \sum_{r\atop\sigma_0-s <r<2}
\textrm{Res}\big[ K(w,x):w=r \big] \frac{f'}{f}(s+r) \\ & + E
\left\{ - K(1-s,x) + \sum_{\rho\in\Z(\Gamma)} K(\rho-s,x) \right\}
\\& + B \Re \left\{-K(1-s+i\gamma_0,x) +
\sum_{\rho\in\Z(\Gamma-i\gamma_0)} K(\rho-s+i\gamma_0,x) \right\}
\\& + C \Re \left\{-K(1-s+i2\gamma_0,x) +
\sum_{\rho\in\Z(\Gamma-i2\gamma_0)} K(\rho-s+i2\gamma_0,x)
\right\},
\end{align*}
where $\sigma_0$ is the abscissa of the line segment of $\Gamma$
crossing the real axis. The only singularity $r$ of $K(w,x)$ in
$[\sigma_0-s,2]$ is at $r=0$ (for $\Gamma(w/2)$ has singularities
only at $0$ and at the negative even integers, whilst
$-1.01<\si_0-s$ excludes the occurrence of the latter),
so the first sum reduces to $2(f'/f)(s)$ in view
of the residuum of $x^w\Gamma(w/2)$ at $w=0$ being exactly 2.
Multiplying by the constant factor $1/2$ in front of it and then
subtracting from the formula the resulting term $\frac{f'}{f}(s)$,
we arrive at
\begin{align}\label{positiveformula}
0 & \leq \sum_{g\in \G} \frac{\Lambda(g)P(\gamma_0 \log |g|)
\left(2-2\exp(-(|g|/x)^2)\right) }{|g|^s} = F(s,x)-\frac{f'}{f}(s)
\notag \\ & = I + E S_0 + B \Re  S_1 + C \Re S_2 ,
\end{align}
where
$$
I:=\frac{1}{2\pi i}
\int\limits_{\Gamma-s} \frac12 \frac{f'}{f}(s+w)K(w,x) dw ,
$$
and
$$
S_j:= -K(1-s+ij\gamma_0,x) +
\sum_{\rho\in\Z(\Gamma-ij\gamma_0)} K(\rho-s+ij\gamma_0,x)\qquad (j=0,1,2).
$$
Using that $s$ is chosen to be real, and that we only need $\Re
S_j$, we can also write $\Re S_j=\Re \overline{S_j}$, where now
$$
\overline{S_j}= -K(1-s-ij\gamma_0,x) +
\sum_{\rho\in\Z(\Gamma+ij\gamma_0)} K(\rho-s-ij\gamma_0,x)\qquad (j=0,1,2).
$$

Recall that the curve $\Gamma$ is constructed in Lemma \ref{l:path-translates} so that
the real part of points on the curve lie between $a:=\frac{b+\theta}{2}$ and $b$.
Therefore, $-1-1/100 \leq a-s \leq \Re w \leq b-s <b-1<0$ bounds occurring values of
$w=u+iv \in \Gamma-s$ within the strip $-1.01\le u=\Re w \le 1-b$.
Stirling's formula for $w=u+iv$ in this strip gives
$|\Gamma(w/2)|\ll \frac1{1-b} (1+|v|)^{u/2-1/4} e^{-\pi v/4}$.

Next we estimate the logarithmic derivatives $\frac{\zeta'}{\zeta}(z+ij\gamma_0)$ for $j=-2,-1,0,1,2$, for arbitrary points $z\in \Gamma$. Using Lemma \ref{l:zzpongamma-c} and $\gamma_0\ge 100$, we are led to
\begin{align*}
\left| \frac{\ze'}{\ze}(s+ij\gamma_0) \right| & \le 5
\frac{1-\theta}{(b-\theta)^{3}} \left(6 \log(|\Im z|+2\gamma_0+5)+60\log(A+\kappa) + 40 \log\frac1{b-\theta}+ 140\right)^2
\\ &\le A_1 \log^2(\gamma_0+|\Im z|) \qquad (A_1:=A_1(\theta,A,\kappa)\quad\textrm{a constant}).
\end{align*}

Applying this bound in estimating the values of $(\zeta'/\zeta)(s+w+ij\gamma_0)$ for $w\in \Gamma-s$ and combining it with the estimation coming from Stirling's formula for $\Gamma(w/2)$ we obtain
\begin{align}\label{I}
I&\ll \int_0^\infty A_1\log^2(\gamma_0+v)x^{b-s} \frac{(1+|v|)^{(b-s)/2-1/4}}{1-b} e^{-v\pi/4} dv
\ll A_2 x^{b-s} \log^2\gamma_0,
\end{align}
with $A_2:=A_2(\theta,A,\kappa)$ a constant.

Next, we estimate the contribution of zeroes in $\overline{S_j}$,
having imaginary part farther from $j\gamma_0$ than $Q$, with a
parameter $Q>1$ to be chosen later. One summand is at most
$$
|K(\beta-s + i(\gamma-j\gamma_0),x)|\ll x^{1-s}
|\gamma-j\gamma_0|^{(1-s)/2-1/4} e^{-\pi|\gamma-j\gamma_0|/4}
\leq x^{1-s}e^{-\pi|\gamma-j\gamma_0|/4},
$$
since for $-1-1/100<\Re w < 1-s $ and $|\Im w|>Q\geq 1$ we have a
uniform Stirling bound.
For the corresponding sums we thus have, by an application of Lemma
\ref{l:Littlewood}
\begin{align*}
\left|\sum_{\rho\in\Z(\Gamma+ij\gamma_0) \atop
|\gamma-j\gamma_0|\geq Q} K(\rho-s-ij\gamma_0,x) \right| & \ll
x^{1-s} \int_{|t-j\gamma_0|>Q} e^{-\pi|t-j\gamma_0|/4} dN(a,t)
\notag \\ & = x^{1-s} \int_Q^{\infty} e^{-\pi u/4}
d\left(N(a,u+j\gamma_0)+N(a,j\gamma_0-u)\right) \notag \\ & \ll
x^{1-s} \left\{ N(a,j\gamma_0+Q)e^{-\pi Q/4}  + \int_Q^{\infty}
e^{-\pi u/4} N(a,u+j\gamma_0)du \right\} \notag \\ & \ll A_3
x^{1-s} (Q+\gamma_0)^2 e^{-\pi Q/4},
\end{align*}
where as before, $A_3$ is a constant depending on the parameters
$\kappa, A$ and $\theta$ only. For $j=0$ we
have for the small zeroes of $S_0$ the estimate
$$
\left|\sum_{\rho\in\Z(\Gamma) \atop |\gamma|<1}
K(\rho-s,x) \right| \leq \frac{A+\kappa}{\kappa(1-\theta)} N(a,1)
x^{1-s} \le A_4 x^{1-s} \qquad (A_4:=A_4(\kappa, A, \theta)~).
$$
according to the separation of zeroes of $\zs$ from $1$, given in
Lemma \ref{l:zkiss}, \eqref{polenozero} and using Lemma \ref{l:Jensen}
\eqref{zerosinH-smallt}, too. Similarly, for the rest of $S_0$ we have,
referring to the uniform Stirling bound here,
$$
\left|\sum_{\rho\in\Z(\Gamma) \atop 1\leq |\gamma|<Q}
K(\rho-s,x) \right| \leq \sum_{q=1}^Q N(a,q,q+1)
x^{1-s} e^{-q\pi/4} \ll A_5 x^{1-s} \quad (A_5:=A_5(\kappa, A, \theta)).
$$
Note that $1-s-i\gamma_0$ and $1-s-2i\gamma_0$ are separated by at
least $\gamma_0>1$ from the singularity of $\Gamma(w/2)$ at $w=0$, whence
we have $|K(1-s-ij\gamma_0)|\ll x^{1-s}$.

Choosing e.g. $Q:=\lceil 3 \log \gamma_0\rceil $ the above estimates applied in \eqref{positiveformula} furnish\footnote{From now on we drop writing $A_k:=A_k(\theta,A,\kappa)$, but all $A_k$ are understood as constants depending on these parameters of Axiom A only.}
\begin{equation}\label{EBCA9}
0\leq - E K(1-s,x) + B S_1' + C S_2' + A_6 \left(x^{b-s} \log^2 \gamma_0+x^{1-s}\right),
\end{equation}
where for $j=1,2$ we denote
\begin{equation}\label{nonnegativereducedsum}
Sj':=\Re \sum_{\rho\in\Z(\Gamma+ij\gamma_0) \atop |\gamma-j\gamma_0| < Q}
K(\rho-s-ij\gamma_0,x) .
\end{equation}

Assume now that $0<\lambda<\lambda_0:=\lambda_0(a):=1-a=\frac23 (1-\theta)$
(so that $\lambda<\lambda_0$ is exactly as is postulated in the theorem).
With such a constant parameter value $\lambda$ we have $1-\lambda>a$, whence
according to \eqref{zerosbetweenone} of Lemma \ref{c:zerosinrange}
the number of zeroes in the discs
$D(1+ij\gamma_0,\lambda)$ ($j=1,2$) cannot exceed
$A_7+\log(2+j\gamma_0)$. Similarly, the number of
summands in both sums $S'_j$ ($j=1,2$) are at most
\begin{equation}\label{numberofterms}
N(a,j\gamma_0-Q,j\gamma_0+Q)\ll  Q
\left(A_8+\log(Q+2\gamma_0)\right) \ll
A_9\log^2 \gamma_0,
\end{equation}
according to Lemma \ref{c:zerosinrange} in view of the choice
of $Q:=[3\log \gamma_0]$.


Since $b=1-(1-\theta)/3$, we have
$x^{b-s}=x^{-(1-\theta)/3+1-s}$. We will now assume that
\begin{equation}\label{X0condition}
x\geq X^*\geq \exp\left( \frac{3}{1-\theta}\log(2\log\gamma_0) \right),
\end{equation}
so that $x^{b-s} \log^2 \gamma_0 \leq x^{1-s}$.
Using this in the above \eqref{EBCA9}, we have for $x\ge X^*$
\begin{equation}\label{intermediate}
0\leq - E K(1-s,x) + B K(\rho_0-s-i\gamma_0,x) + B S_1"
+ C S_2' + 2 A_6 x^{1-s},\notag
\end{equation}
where $S_1"$ is essentially $S_1'$, save the contribution of the
known zero at $\rho_0$.

Next we separate the contribution of the zeroes inside ${\mathcal
D}:= D(1+i\gamma_0,\lambda) \cup D(1+i2\gamma_0,\lambda)$. Here we
use the estimate, proved in \cite{Mont} and also in \cite{Ram},
that for any $z$ subject to $-3/2<\Re z < 0$ we have uniformly
$\Re x^z \Gamma(z/2) \ll x^{\Re z} \log x$. In view of this, and
our choice of $1<s<1.01$, the contribution of these close zeroes
is at most a constant times
$$
S^{\star}:=\sum_{\rho\in\DD\setminus\{\rho_0\}} x^{\beta-s}\log x.
$$
Collecting the above estimates we obtain
$$
E K(1-s,x) - B K(\beta_0-s,x) - B S_1^{\star}
- C S_2^{\star} - 2 A_6 x^{1-s} \ll S^{\star},
$$
where now the sums $S_j^{\star}$ denote the same as $S_j'$, but the
summands chosen only from
$$
\{\rho\in\Z(\Gamma-ij\gamma_0)~:~|\gamma-j\gamma_0|<
Q\}\setminus D(1+ij\gamma_0,\lambda).
$$
Note that all terms in the
sums $S^{\star}_j$ ($j=1,2$) have the form $K(z,x)$ with
$|z|>\lambda$.

Now it may be clear that any choice of the coefficients $B,C,D$,
hence $E$, which satisfy $E<B$, suffices. So we choose $B:=4$,
$C:=1$ and $D:=4$ to obtain $E=3$ (and thus $P(u)=3+4\cos u + \cos
(2u)$). After multiplying by $x^{s-1}$ we are led to
\begin{equation}\label{closetointegral}
3\Gamma\left(\frac{1-s}{2}\right)-4x^{\beta_0-1}\Gamma\left(\frac{\beta_0-s}{2}\right)
-4 S_1^{\star\star} -  S_2^{\star\star}-2 A_6
\ll\sum_{\rho\in\DD\setminus\{\rho_0\}} x^{\beta-1}\log x,
\end{equation}
where
\begin{equation}\label{Sdublestar}
Sj^{\star\star}:=\Re   \sum_{\rho\in\Z(\Gamma+ij\gamma_0)
\atop |\gamma-j\gamma_0| < Q,~|\rho-ij\gamma_0|>\lambda}
x^{\rho-ij\gamma_0-1}\Gamma\left(\frac{\rho-s-ij\gamma_0}{2}\right).
\end{equation}
We want to chose our parameters $x$ and $y:=\log x$ so that we will have
\begin{equation}\label{parameters}
\frac{4}{1-\theta} \log\log \gamma_0 < y: = \log x \leq \frac{0.1}{1-\beta_0},
\end{equation}
meeting the requirement in \eqref{X0condition}, too, and further that it will hold
\begin{equation}\label{sbeta}
(1<)~1+10(1-\beta_0)< s < 1.01.
\end{equation}
Note that in order to have a nonempty interval for $y$ in the first
condition it suffices to have $1-\beta_0< \frac{1-\theta}{40\log\log\gamma_0}$ which is
assumed among the conditions of the statement of Theorem \ref{th:Ramachandra}.
Further, to have a nonempty interval for $s$ in the
second condition it suffices to have $1-\beta_0 <0.001$, also
guaranteed by the condition $\beta_0>0.999$ of Theorem \ref{th:Ramachandra}.

The above conditions allow an interval for the choice of $x$. We
will now take a subinterval, so that $\frac{4}{1-\theta}\log\log \gamma_0\leq Y_0:=\log X_0<
Y_1:=\log X_1 \leq 0.1/(1-\beta_0)$, and consider admissible values
$y:=\log x = Y_0+u_1+\dots+u_m$ with $0\leq u_j\leq d$, where
$d:= 2 e^2/\lambda$ is a constant, now depending also on $\lambda$,
and $m\in \NN$ is an integer parameter. Moreover, we want to choose $m:=[Y_0]$. Then
the above parametrization of $y=\log x$ will run between $Y_0$ and
$Y_0+md$, whence $Y_1=Y_0+md\in (Y_0+d(Y_0-1),(1+d)Y_0]$ must be below the
bound $0.1/(1-\beta_0)$ for $\log X_1$. We thus
require $(1+d)Y_0 \leq 0.1/(1-\beta_0)$. Whenever such a $Y_0\ge \frac{4}{1-\theta}\log\log \gamma_0$ is
chosen, the corresponding $m$ and $Y_1=Y_0+md$ satisfy the
necessary bounds, hence the interval $[X_0,X_1]$ is admissible.
So from this point on we can proceed with our argument only in case
\begin{equation}\label{yzeroint}
\frac{4}{1-\theta}\log\log \gamma_0  \leq Y_0 \le \frac{1}{1+d} \frac{0.1}{1-\beta_0}.
\end{equation}
The left hand side inequality is assumed (with $Y$ written in place of $Y_0$ here)
among the conditions of Theorem \ref{th:Ramachandra}.
Now, if the right hand side inequality here fails to hold, then we necessarily have
$Y_0 > \frac{1}{1+d} \frac{0.1}{1-\beta_0} = \frac{\lambda}{(\lambda+e^2)10(1-\beta_0)}$,
hence $Y_0(1-\beta_0) > \frac{\lambda}{160}$, so that for any constant $c_0\ge 160 $ in the statement
\eqref{eq:Ramachandrastatement} the right hand side becomes negative.
Therefore, for such parameter values $Y=Y_0$ there remains nothing to prove, and
it remains to derive the assertion in case \eqref{yzeroint} holds.

Observe that for any values of $\beta_0$ and $s$ satisfying \eqref{sbeta} we have
\begin{align*}
3\Gamma\left(\frac{1-s}{2}\right) & -
4x^{\beta_0-1}\Gamma\left(\frac{\beta_0-s}{2}\right)
\\ &=-\frac{6}{s-1} \Gamma\left(\frac{3-s}{2}\right) +
\frac{8}{s-\beta_0} e^{-(1-\beta_0)\log x} \Gamma\left(\frac{2+\beta_0-s}{2}\right)
\\ &\geq \left\{\frac{8}{s-1+1-\beta_0} e^{-0.1} - \frac{6}{s-1} \right\}
\Gamma\left(\frac{3-s}{2}\right)
\\ &\geq \left\{\frac{8}{(s-1)1.1e^{0.1}}
- \frac{6}{s-1} \right\} \Gamma\left(\frac{3-s}{2}\right), \notag
\end{align*}
because $\Gamma$ is decreasing around $1$ and thus
$\Gamma\left(\frac{2+\beta_0-s}{2}\right)>\Gamma\left(\frac{3-s}{2}\right)$.
So in all we get from \eqref{closetointegral}
$$
3\Gamma\left(\frac{1-s}{2}\right)-
4x^{\beta_0-1}\Gamma\left(\frac{\beta_0-s}{2}\right)>\frac{1}{s-1}.
$$
Using also $Y_0=\log X_0\leq \log x\leq Y_1\leq (1+d)Y_0$ we obtain
from \eqref{closetointegral} and the above
\begin{equation}\label{sminus1X0}
\frac{1}{s-1}-2 A_6 -4 S_1^{\star\star} -  S_2^{\star\star}
\ll\sum_{\rho\in\DD\setminus\{\rho_0\}} x^{\beta-1}\log x
\leq (1+d)Y_0 \sum_{\rho\in\DD\setminus\{\rho_0\}}
X_0^{\beta-1}.
\end{equation}
Since we have this inequality for all $x\in[X_0,X_1]$, i.e. for
all values of ${\bf u}:=(u_1,\dots,u_m)\in[0,d]^m$, we can
average it on the left hand side with respect to all the $u_k$.
For the general term of the sums $S_j^{\star\star}$ we obtain
\begin{align*}
\int_0^d\int_0^d\dots\int_0^d & \left\{ \Re X_0^{\rho-ij\gamma_0-1}
e^{(u_1+\dots+u_m)(\rho-ij\gamma_0-1)}
\Gamma\left(\frac{\rho-s-ij\gamma_0}{2}\right) \right\}
du_1 \dots du_m \\ & \ll 2^m \frac{1}{\lambda}
\frac{1}{|\rho-ij\gamma_0-1|^m} \leq \left( \frac{2}{\lambda}
\right)^{m+1},
\end{align*}
because $|\rho-ij\gamma_0-1|\geq \lambda$ as $\rho$ lies outside
$D(1+ij\gamma_0,\lambda)$ and $|\Gamma\left(\frac{\rho-s-ij\gamma_0}{2}\right)|
\ll 1/\lambda$.

Taking into account \eqref{numberofterms} and the total volume $d^m$ of the
cube $[0,d]^m$, we obtain from \eqref{sminus1X0} with an absolute constant $C^{\star}$ the inequality
$$
\frac{1}{s-1}-2A_6-  C^{\star} A_9 \log^2 \gamma_0 ~d \left( \frac{2}{d \lambda}
\right)^{m+1}
\leq (1+d)Y_0 \sum_{\rho\in\DD\setminus\{\rho_0\}}
X_0^{\beta-1}.
$$
Now in view of $d:=2e^2/\lambda$, we have
$ \left( \frac{2}{d \lambda} \right)^{m+1} =
\exp(-2(m+1))\leq \exp(-2Y_0)\leq \log^{-2} \gamma_0$, hence writing
in the value of $d$ and putting $C":=2e^2 C^{\star}$ we arrive at
$$
\frac{1}{s-1}-2 A_6 - \frac{C"A_9}{\lambda} \ll \frac{1}{\lambda} Y_0 \sum_{\rho\in\DD\setminus\{\rho_0\}}
X_0^{\beta-1}.
$$
Clearly our best choice here is to choose $s$ as small as
possible. We must meet the conditions \eqref{parameters}, so the
smallest admissible value is $s=1+10(1-\beta_0)$. This choice
yields
$$
\frac{{\lambda}}{Y_0 (1-\beta_0)}-\frac{20 A_6\lambda +10 C" A_9}{Y_0}
\ll \sum_{\rho\in\DD\setminus\{\rho_0\}} e^{Y_0(\beta-1)}.
$$
Applying $\lambda <1$ and using $Y_0\ge A_{10}$ with a sufficiently large value of
the constant $A_{10}$ now gives the asserted inequality
\eqref{eq:Ramachandrastatement} for $Y_0$ in place of $Y$.

The result is proved.
\end{proof}

\section{A preview of further work}\label{sec:preview}

Let us motivate our detailed study of the distribution of zeroes of the Beurling zeta function by recalling three rather sharp, essentially optimal results, known to hold \emph{for the classical case} of natural numbers and primes and the Riemann zeta function.

The first one is the essentially final answer to a classic problem of Littlewood \cite{Littlewood} as to what oscillation could be "caused" by having a $\zeta$-zero $\rho$? For the result and its possible optimality (in particular as regards the somewhat surprising value of the constant $\pi/2$ in it) see \cite{RevAA}.
\begin{theorem}[R\'ev\'esz]\label{th:onezeroosci} Let $\ze(\rho)=0$
with $\rho=\beta+i\gamma$ be a zero of the Riemann zeta function.
Then for arbitrary $\ve>0$ we have for some suitable, arbitrarily large values
of $x$ the lower estimate $|\Delta(x)|\geq
(\pi/2-\ve) \frac{x^{\beta}}{|\rho|}$.
\end{theorem}

Denote by $\eta(t):(0,\infty)\to (0,1/2)$ a nonincreasing function and consider the domain
\begin{equation}\label{eq:etazerofree}
\DD(\eta):=\{ s=\sigma+it \in\CC~:~ \sigma>1-\eta(t),~ t>0\}.
\end{equation}
Following Ingham \cite{Ingham} and Pintz \cite{Pintz1,Pintz2}
consider also the derived function (the \emph{Legendre transform} of $\eta$ in logarithmic variables)
\begin{equation}\label{omegadef}
\omega(x):=\omega_{\eta}(x):=
\inf_{y>1} \left(\eta(y)\log x+\log y\right).
\end{equation}

Then the following are sharp results of Pintz, see \cite{Pintz1, Pintz2}, sharp in the ultimate sense that iterative applications of them give back the original estimates without any loss in the constants in the exponents.

\begin{theorem}[Pintz]\label{th:domainesti} Assume that there is no
zero of the Riemann $\zeta$ function in $\DD(\eta)$. Then for arbitrary $\ve>0$
we have $$\Delta(x)=O(x\exp(-(1-\ve)\omega(x)).$$
\end{theorem}

\begin{theorem}[Pintz]\label{th:domainosci}
Conversely, assuming that there are infinitely many zeroes within
the domain \eqref{eq:etazerofree}, we have for any $\ve >0$
the oscillation estimate $\Delta(x)=\Omega(x\exp(-(1+\ve)\omega(x))$.
\end{theorem}

These results, in their original proofs relied on particular things which are generally not available for the Beurling zeta functions. Therefore, it is unclear how much of these precise relations can as well be stated for the distribution of Beurling primes? Our aim with the present series is to prove analogously sharp results for the Beurling case eventually. The above cited theorems are much sharper than everything currently known for the Beurling case--so that extending them would indeed mean a considerable advance. That is our long range aim with this series.

\end{document}